\documentclass{article}

\usepackage[utf8]{inputenc} 
\usepackage{url}

\usepackage{amsmath,amsfonts,amssymb}
\usepackage{amsthm}
\usepackage{graphicx}
\usepackage{enumerate}
\usepackage{multirow,bigdelim}
\usepackage{mathtools}
\usepackage{soul} 
\usepackage{pstricks}
\usepackage[margin=3cm]{geometry}
\usepackage{array}
\usepackage[normalem]{ulem}
\usepackage{url}
\usepackage{hyperref}
\usepackage{enumerate}
\usepackage{enumitem}
\usepackage{mathtools}
\hypersetup{
	colorlinks,
	linkcolor={red!50!black},
	citecolor={blue!50!black},
	urlcolor={blue!80!black}
}

\newcommand{\SOC}[2]{{\mathcal{L}^{#2} _{#1}}}

\newcommand{\interior}{\mathrm{int}\,}

\newcommand{\closure}{\mathrm{cl}\,}

\newcommand{\conicHull}{\mathrm{cone}\,}

\newcommand{\norm}[1]{\lVert{#1}\rVert}
\newcommand{\inProd}[2]{\langle #1 , #2 \rangle }

\newcommand{\stdCone}{ {\mathcal{K}}}

\newcommand{\stdFace}{ \mathcal{F}}

\newcommand{\stdInt}{ {e}}

\renewcommand{\Re}{\mathbb{R}}

\newcommand{\tr}{\mathrm{tr}}

\newcommand{\T}{*} 
\newcommand{\alg}{\mathcal{A}}

\newcommand{\UT}{\mathcal{T}}

\newcommand{\stdHC}{ {\mathcal{K}(\alg)}}

\DeclarePairedDelimiter\abs{\lvert}{\rvert}%
\newtheorem{definition}{Definition}

\newtheorem{proposition}[definition]{Proposition}

\newtheorem{theorem}[definition]{Theorem}
\newtheorem*{proposition*}{Proposition}
\theoremstyle{remark}

\title{The $p$-cones in dimension $n \geq 3$ are not homogeneous when $p\neq 2$}
\author{
	Masaru Ito%
		\thanks{Department of Mathematics, College of Science and Technology, Nihon University,
			1-8-14 Kanda-Surugadai, Chiyoda-Ku, Tokyo 101-8308, Japan 
			(\texttt{ito.m@math.cst.nihon-u.ac.jp)}.}
	\and
	Bruno F. Louren\c{c}o%
	\thanks{Department of Computer and Information Science, Faculty of Science and Technology, Seikei
		University, 3-3-1 Kichijojikitamachi, Musashino-shi, Tokyo 180-8633, Japan.
		(\texttt{lourenco@st.seikei.ac.jp)}}
}

\begin{document}
\maketitle
\begin{abstract}
Using the $T$-algebra machinery we show that, up to linear isomorphism, the only strictly convex homogeneous cones 
in $\Re^n$ with $n \geq 3$ are the $2$-cones, also known as Lorentz cones or second order cones.
In particular, this shows that the $p$-cones are not homogeneous when $p\neq 2$, $1 < p <\infty$ and 
$n\geq 3$, thus answering a problem proposed by Gowda and Trott. 

\end{abstract}
\section{Introduction}
We prove that if $p\neq 2$ and $1< p < \infty$, then the $p$-cones $\SOC{p}{n}$ in the 
$n$-dimensional vector space $\Re^n$ are not homogeneous when $n \geq 3$. This solves a problem posed by Gowda and Trott in Section 6 of \cite{GT14}, where they proved the non-homogeneity of $\SOC{1}{n}$ and its dual $\SOC{\infty}{n}$.
In fact, we will prove a more general statement and show 
that, up to isomorphism, the only strictly convex homogeneous cones are the $2$-cones, that is, the Lorentz cones.

Recall that a closed proper cone is said to be symmetric if it is self-dual and homogeneous.
However, an often overlooked point is that \emph{self-duality} is a concept that depends on the choice of 
inner product.
Recently, in a published article,  we saw  an attempt to equip $\Re^n$ with an inner product (depending on $p$) so 
that $\SOC{p}{n}$ becomes self-dual. In the same article, it was claimed that $\SOC{p}{n}$ 
is homogeneous, thus showing that it is a symmetric cone under an appropriate inner product. For a discussion of its flaws, we refer to the paper by Miao, Lin, Chen \cite{MLC17}.
The result we prove here implies, in particular, that for $n \geq 3$ and $p \neq 2$, 
$\SOC{p}{n}$ is never a symmetric cone. This, however, does not rule out the possibility 
of $\SOC{p}{n}$ being self-dual under an appropriate inner product, which also seems to be an unsettled problem. 
We remark that it can be  shown that  $\SOC{p}{n}$ is not self-dual under  ``reasonable'' inner products, see Section 3 in \cite{MLC17}.

Another motivation for this work comes from the fact that optimization problems involving the 
$p$-norms or the $p$-cones are sometimes refereed to as \emph{nonsymmetric optimization problems}, see, for example, \cite{NY12,SY15}.
Although not as popular as optimization over second order cones, problems over $p$-cones do appear in the literature occasionally \cite{XY00,NY12,SY15} and it might be fair to say that they are an important 
part of the conic linear programming landscape.
 But as far as we know, no previous work has actually went through the trouble of checking whether it is indeed the case  that $\SOC{p}{n}$ is non-symmetric, non-homogeneous or non-self-dual.
Had  $\SOC{p}{n}$ turned out to be symmetric after an appropriate change of inner product, this fact alone would have very interesting algorithmic repercussions, so it is somewhat surprising that this was left unknown for a relatively long time.
In the same vein, had $\SOC{p}{n}$ turned out to be   homogeneous but non-symmetric, following several works on homogeneous cones \cite{GO98,TX01,CH03,CH09}, this would also have remarkable consequences.

%We remark that optimization problems involving the $p$-norms or the $p$-cones are sometimes 
%referred as \emph{nonsymmetric} optimization and this stands in contrast to the concept of optimization over 
%\emph{symmetric} cones, such as $\SOC{2}{n}$, the nonnegative orthant $\Re^n_+$ or 
%the positive semidefinite matrices $\PSDcone{n}$. 
%We recall that a closed convex cone is said to be \emph{symmetric} if it is both self-dual and homogeneous. As it is well-known, for $n \geq 3$, $p \neq 2$,  $\SOC{p}{n}$ is not self-dual \emph{under the Euclidean inner product}.
%However, the notion of self-duality is entirely dependent on the inner product so a cone might as well become self-dual when an appropriate inner product is introduced. 
%We would like to suggest the avoidance of the term \emph{nonsymmetric} when a cone has not been conclusively shown to be nonsymmetric, which until this paper, seemed to be the case of the $p$-cones.

Reading the literature on $p$-cones, one might get the impression that the great 
hurdle is the (apparent) lack of self-duality.
However, not only it seems to be unknown whether the $p$-cones are never self-dual but we think that the real culprit is the lack of homogeneity. 
Consider, for instance, interior point methods.
Their performance is highly dependent on the so-called \emph{complexity parameter} of self-concordant barriers and for $\SOC{2}{n}$, the complexity of an optimal barrier is $2$, regardless of $n$. In contrast, all the barriers 
for $\SOC{p}{n}$ with known closed-form expressions have complexity parameter proportional to $n$ and the best one so far has complexity $4n$, see Section 1 in the work by Nesterov \cite{NY12}. 
Recently, Hildebrand proved  \cite{H14} that a regular $n$-dimensional convex cone admits a barrier with parameter not exceeding $n$, which implies that  $\SOC{p}{n}$ also has a barrier with parameter $n$, although it could be hard to compute it or to obtain a closed form expression.
If $\SOC{p}{n}$ were homogeneous of rank $r$, Proposition \ref{prop:rank2} would imply $r \leq 2$. Then, 
we would readily have access to a barrier 
with parameter  not larger than $2$, due to a result by G\"uler and Tun\c{c}el, see Theorem 4.1 in \cite{GO98} and the related paper by Truong and Tun\c{c}el \cite{TT03}. 
As far as we know, a barrier for $\SOC{p}{n}$ not depending on $n$ has never appeared previously in the literature and would be a truly remarkable object. 
It seems that it is still an open problem to determine the optimal barrier parameter for $\SOC{p}{n}$. 

The outline of the proof is simple. Using the theory of $T$-algebras, we first check that the closure 
of an homogeneous cone of rank $r \leq 2$ is isomorphic to either $\{0\}$, $\Re_+$, $\Re_+^2$ or 
$\SOC{2}{n}$. Then, we argue that if $\SOC{p}{n}$ were to be a homogeneous cone then its rank 
would be less or equal than $2$, which is the main missing piece we supply in this article. 
Now, Gowda and Trott proved in \cite{GT14} that $\SOC{2}{n}$ and 
$\SOC{p}{n}$ are not isomorphic for $n \geq 3$ and $p \neq 2$. So for $n\geq 3$ and $1 < p < \infty$, $\SOC{p}{n}$ cannot possibly be homogeneous, since it is not isomorphic to the lower dimensional cones 
$\{0\}$, $\Re_+$, $\Re_+^2$.

\section{Preliminaries on convex cones}
A \emph{convex cone} is a set $\stdCone$ contained in some real vector space $\alg$, 
such that $\alpha x + \beta y \in \stdCone$, for all $x,y \in \stdCone$ and $\alpha,\beta \in \Re_+$. 
If $\alg$ is equipped with some inner product $\inProd{\cdot}{\cdot}$ we 
can define the dual cone of $\stdCone$ as $\stdCone^* = \{x \in \alg \mid \inProd{x}{y} \geq 0, \forall y \in \stdCone \}$. 
 We will write $\interior \stdCone, \closure \stdCone, \dim \stdCone$ for the interior, closure  and dimension of $\stdCone$, respectively. 

A convex cone $\stdCone$ is said to be \emph{pointed} if $\closure \stdCone \cap -\closure \stdCone =\{0\}$ and 
it is said to be \emph{full-dimensional} if $\interior \stdCone \neq \emptyset$. Note that all 
convex cones can be made to be full-dimensional if we substitute the underlying space by the span of $\stdCone$.
An \emph{automorphism of $\stdCone$} is a linear bijection $Q$ such that 
$Q\stdCone = \stdCone$. Then, $\stdCone$ is said to be \emph{homogeneous} if it is a 
full-dimensional pointed convex cone such that its group of automorphisms acts transitively on $\interior \stdCone$. This means that for every $x,y \in \interior \stdCone$, there is an automorphism  $Q$ of $\stdCone$ for 
which $Q(x) = y$.

In some works on convex cones it is common to consider  \emph{open convex cones}, that is, convex cones 
satisfying $\stdCone = \interior \stdCone$. 
In fact, the definition of ``convex cone'' 
by Vinberg included the requirement that $\stdCone$ should be open. In optimization, however, it is common to 
consider \emph{closed convex cones}. Suppose that $\stdCone$ is full-dimensional, then $\interior \stdCone = \interior (\interior \stdCone) = \interior (\closure \stdCone)$.  Therefore, for the study of homogeneity, it does not matter whether we study $\interior \stdCone$, $\stdCone$ or $\closure \stdCone$. 

Two convex cones $\stdCone _1, \stdCone _2 $ are said to be \emph{isomorphic} if there 
is a linear bijection $Q$ such that $Q \stdCone _1 = \stdCone _2$. Note that if 
$\stdCone_1$ and $\stdCone _2$ are full-dimensional, then $\closure \stdCone _1$ and $\closure \stdCone _2$ are isomorphic 
if and only if $\interior \stdCone _1$ and $\interior \stdCone _2$ are isomorphic.

Let $C$ be a convex set. A convex set  $\stdFace \subseteq C$ is said to be 
a \emph{face} of $C$ if the following condition holds: if $x,y \in C$ and $\alpha x + (1-\alpha)y \in \stdFace$ for some $0 < \alpha < 1$, then $x,y \in \stdFace$. A face $\stdFace$ is said to be \emph{proper} if $\stdFace \neq C$. An \emph{extreme point}  is a face consisting of a single point. 

\subsection{Strictly convex cones}
A compact convex set $C$ with nonempty interior is said to be \emph{strictly convex}  if every proper face of $C$ is an extreme point. Similarly, a pointed closed convex cone $\stdCone$ with nonempty interior is said to be \emph{a strictly convex cone} if every proper face of $\stdCone$ has dimension $0$ or $1$.

A norm $\norm{\cdot}$ on a real vector space is said to be a \emph{strictly convex norm} if
$$\norm{x+y}<2 ~\text{ whenever }~ \norm{x}=\norm{y}=1,~~ x \ne y$$
or, equivalently,
$$
\norm{\alpha x + (1-\alpha)y} < 1  ~\text{ whenever }~ \norm{x}=\norm{y}=1,~~ x \ne y,~~\alpha \in (0,1).
$$
The relations between these notions is as follows.
\begin{proposition}\label{prop:strict-conv}
Let $\norm{\cdot}$ be a norm on a real vector space $\alg$.
Then, the following are equivalent.
\begin{enumerate}[label=({\it \roman*})]
\item $\norm{\cdot}$ is a strictly convex norm.
\item $B = \{x \in \alg \mid \norm{x} \leq 1\}$ is a strictly convex set.
\item $\stdCone = \{(t,x) \in \Re \times \alg \mid \norm{x} \leq t \}$ is a strictly convex cone.
\end{enumerate}
\end{proposition}

\begin{proof}
Note that $B$ is strictly convex if and only if $\alpha x + (1-\alpha)y \in \interior B$ whenever $x,y \in B$ and $\alpha \in (0,1)$. This proves the equivalence {\it (i)} $\Leftrightarrow$ {\it (ii)}.
 
Now, it is straightforward to check that, for a compact convex set $C$, the map
$$
\stdFace \mapsto \conicHull(\{1\}\times \stdFace) = \{(t,tx) \mid t \geq 0,~ x \in \stdFace\}
$$
is a bijection from the set of faces of $C$ onto the set of nonzero faces of $\conicHull(\{1\}\times C)$.
Hence, the equivalence {\it (ii)} $\Leftrightarrow$ {\it (iii)} immediately follows
because $\stdCone = \conicHull(\{1\} \times B)$ holds.
\end{proof}
We are particularly interested in the case of $p$-norms that is,
$$
\norm{x}_{p} = \left( \sum _{i=1}^n \abs{x_i}^p \right)^{1/p}
$$
for $p \in [1,\infty)$ and $\norm{x}_{\infty} =  \max _{1\leq i \leq n} \abs{x_i}$.
Then, the $p$-cone is defined as
$$\SOC{p}{n}  = \{(t,x) \in \Re\times \Re^{n-1} \mid t \geq \norm{x}_p \}.$$
It follows by Proposition \ref{prop:strict-conv} that only the $p$-cones $\SOC{p}{n}$ for $1 < p < \infty$ are strictly convex, since those are the values that correspond to strictly convex norms.

\section{$T$-algebras}
$T$-algebras were proposed by Vinberg \cite{V63} as a natural framework for the study of homogeneous convex cones. 
For a more recent treatment, including its connections to optimization, see the work of Chua \cite{CH03,CH09}.
We recall that an \emph{algebra} is a vector space $\alg$ over some field $\mathbb{K}$ such that  $\alg$ is equipped with a product $\times:\alg\times \alg \to \alg$ that satisfies
\begin{align*}
(a+b)\times c & = a\times c + b\times c\\
c\times (a+b) & = c\times a + c\times b\\
(\alpha a)  \times (\beta b) & = (\alpha \beta) a \times b,
\end{align*}
for all $a,b,c \in \alg, \alpha,\beta \in \mathbb{K}$.
In the cases we discuss here, we will always consider the real number field $\Re$ and write 
$a\times b = ab$, for $a,b \in \alg$. Then, a \emph{matrix algebra of rank $r$} is an algebra over $\Re$ that is equipped with a decomposition as a direct sum $\alg = \bigoplus _{i,j=1}^r \alg _{ij}$ where the $\alg _{ij}$ are subspaces  satisfying the following properties:
\begin{align*}
\alg _{ij} \alg _{jk} & \subseteq \alg _{ik} \\
\alg _{ij} \alg _{kl} & = \{0\} \qquad \text{if } j \neq k.  
\end{align*} 
This decomposition is called a \emph{bigradation}.
Therefore, in  a matrix algebra we can represent an element $a \in \alg$ as a generalized 
matrix $a = (a_{ij})_{i,j=1}^r$, where $a_{ij} \in \alg _{ij}$, for all $i,j$. With that, the multiplication in $\alg$ follows the usual matrix multiplication rules $(ab)_{ij} = \sum _{k=1}^r a_{ik} b_{kj}$.

A \emph{matrix algebra with involution} is a matrix algebra equipped with a linear bijection $\T:\alg \to \alg$ such that
$$a^{\T \T} = a,~~ (ab)^\T = b^\T a^\T~~ \text{and} ~~A_{ij}^\T = A_{ji}~~ \text{ for all }  i,j.$$
With that,  we have $(a^\T)_{ij} = a_{ji}^\T$.

Finally, a \emph{$T$-algebra} is a matrix algebra with involution satisfying the following properties, see 
Definition 4 in \cite{CH09}.
\begin{enumerate}[label=({\it\roman*})]
\item For each $i$,  $\alg _{ii}$ is a subalgebra isomorphic to $\Re$. \label{ax:1}
\end{enumerate}
Let $\rho _i: \alg_{ii} \to \Re$ denote the algebra isomorphism and 
let $\stdInt _i$ denote the unit element in $\alg _{ii}$, \textit{i.e.}, 
the element satisfying $\rho _i(\stdInt _i) = 1$.
Furthermore, define the function $\tr: \alg \to \Re$ by $\tr(a) := \sum _{i=1}^r \rho _i (a_{ii})$.
\begin{enumerate}[label=({\it \roman*})]
 \setcounter{enumi}{1}
\item For all $a \in \alg$  and all $i,j \in \{1, \ldots, r\}$ we have 
$\stdInt _i a_{ij} = a_{ij}$ and $a_{ji}\stdInt _i = a_{ji}.$ \label{ax:2}
\item For all $a \in \alg$  and all $i,j \in \{1, \ldots, r\}$, we have  $\rho_i(a_{ij}b_{ji}) = \rho _j(b_{ji}a_{ij})$. \label{ax:3}
\item For all $a,b,c \in \alg$ and $i,j,k \in \{1, \ldots, r\}$, we have $a_{ij}(b_{jk}c_{ki}) = (a_{ij}b_{jk})c_{ki}$. \label{ax:4}
\item For all $a \in \alg$ and $i,j\in \{1, \ldots, r\}$, we have $\rho _i(a_{ij}^* a_{ij}) \geq 0$, with equality if and only if $a_{ij} = 0$. \label{ax:5}
\item For all $a,b,c \in \alg$ and $1 \leq i\leq j \leq k \leq l \leq r$, we have $a_{ij}(b_{jk}c_{kl}) = (a_{ij}b_{jk})c_{kl}$. \label{ax:6}
\item For all $a,b \in \alg$, $1\leq i\leq j \leq k \leq r$ and $1 \leq l\leq k \leq r$,  we have $a_{ij}(b_{jk}b_{lk}^*) = (a_{ij}b_{jk})b_{lk}^*$. \label{ax:7}
\end{enumerate}

For a $T$-algebra $\alg$ of rank $r$, we write $\UT$ for the set of ``upper-triangular matrices'' in $\alg$, i.e., 
$\UT = \{a \in \alg \mid a_{ij} = 0~\text{ if }~1\leq j < i \leq r \}$. We define
$\UT_{+} = \{a \in \UT \mid \rho_i(a_{ii}) \geq 0~\text{ if }~ 1 \leq  i\leq r \}$ and
$\UT_{++} = \{a \in \UT \mid \rho_i(a_{ii}) > 0~\text{ if }~ 1 \leq i  \leq r \}$. With that, 
we define the convex cone associated to the $T$-algebra $\alg$ as
$$
\stdHC = \{tt^* \mid t \in \UT_{++} \}.
$$
We define an inner product over $\alg$ by taking 
$\inProd{x}{y} = \tr(x^*y)$.
 We also have
$$
\closure \stdHC = \{tt^* \mid t \in \UT_{+} \},
$$
see the remarks before Proposition 1 in \cite{CH09}. Vinberg proved in \cite{V63}
the following landmark result.
\begin{theorem}\label{theo:vin}
Let $\stdCone$ be an open homogeneous convex cone. Then, there is a 
$T$-algebra $\alg$ for which $\stdHC = \stdCone$. Conversely, if 
$\stdHC = \stdCone$ for some $T$-algebra $\alg$, then $\stdCone$ is 
an open homogeneous convex cone.
\end{theorem}
Following  Theorem \ref{theo:vin}, we define the rank of a homogeneous convex cone as 
the rank of the underlying algebra. This is well-defined because if $\alg$ and 
$\alg'$ are two $T$-algebras such that $\interior \stdCone = \stdHC = \stdCone(\alg')$, 
then $\alg$ and $\alg'$ must be isomorphic due to Theorem 4 in Chapter 3 of \cite{V63}.
We now state a few elementary observations about diagonal elements. 
\begin{proposition}\label{prop:diag}
Let $\alg$ be a $T$-algebra and $a \in \alg, t \in \UT_{+}$, then for every $i$
\begin{enumerate}[label=({\it \roman*})]
	\item $a_{ii}^* = a_{ii}$,
	\item $a_{ii} = \rho _i(a_{ii}) \stdInt_i$,
	\item $\rho _i ((tt^*)_{ii}) \geq 0$.
\end{enumerate}
\end{proposition}
\begin{proof}	
\begin{enumerate}[label=({\it \roman*})]
	\item The restriction of the involution $*$ to $\alg _{ii}$ becomes an automorphism of $\alg_{ii}$. 
	Since 	$\alg_{ii}$ is isomorphic to $\Re$ and the only 
	automorphism of $\Re$ is the identity map, we conclude that the restriction 
	of $*$ to $\alg_{ii}$ must be the identity map as well.
	\item The map $\rho _i$ is an algebra isomorphism, so it satisfies 
	$\rho _i(\alpha x) = \alpha \rho _i(x)$ for every $\alpha \in \Re$ and 
	$x \in \alg _{ii}$. Since $\rho _i ( \rho _i(a_{ii}) \stdInt_i) = \rho _i(a_{ii}) 1$ and 
	$\rho _i$ is a bijection, we conclude that $a_{ii} = \rho _i(a_{ii}) \stdInt_i$.
	\item Note that $(tt^*)_{ii} = \sum _{j=i}^r t_{ij} t_{ij}^*$. Therefore, 
	$\rho _i ((tt^*)_{ii}) =  \sum _{j=i}^r \rho _{i}(t_{ij} t_{ij}^*) $. From 
	Axiom \ref{ax:5}, every term inside the summation is nonnegative, so that 
	$\rho _i ((tt^*)_{ii}) \geq 0$.
\end{enumerate}

\end{proof}

\section{Main result}
We will now gather a few results that will allow us to prove our main result. 
Recall that if $\stdCone$ is a closed convex cone and $y \in \stdCone^*$ then 
$\stdCone \cap \{y\} ^\perp = \{x \in \stdCone \mid \inProd{x}{y} = 0 \}$ is always a nonempty 
face of $\stdCone$.
\begin{proposition}\label{prop:h_face}
The closure of a homogeneous convex cone $\stdCone$ of rank $r \geq 1$ contains a proper face of dimension 
at least $r-1$.	
\end{proposition}	
\begin{proof}
If $r = 1$, we take $\{0\}$ as the desired face. So suppose that $r\geq 2$.
Consider a $T$-algebra $\alg$ of rank $r$  such that $\interior \stdCone = \stdHC$. We will prove 
that $(\closure \stdHC) \cap \{\stdInt _1 \}^\perp$ has dimension greater or equal 
than $r-1$.
We first 
argue that $\stdInt _1 \in \stdHC ^*$. 
Let $x \in \stdHC$, then $(\stdInt _1 x)_{11} = x_{11}$. 
By Proposition \ref{prop:diag}, we have $\rho _1(x_{11}) \geq 0$, therefore 
$\inProd{\stdInt _1}{x} = \tr (\stdInt _1 x) = \rho _1(x_{11}) \geq 0$. 

We have $\closure \stdCone = \closure \stdHC = \{ tt^* \mid t \in \UT_{+}\}$. For all 
$i$, we have $\stdInt _i \in \UT_{+}$, so that $\stdInt _i \stdInt _i^* = \stdInt _i \stdInt _i = \stdInt _i \in \closure \stdCone$. Since $\alg$ is a matrix algebra, $\stdInt _i \stdInt _j = 0$ if 
$i \neq j$. Therefore, $\inProd{\stdInt _i}{\stdInt _j} = 0$ for $i \neq j$. This shows 
that $\{\stdInt _2, \ldots, \stdInt _r \} \subseteq (\closure \stdCone )\cap \{\stdInt_1 \}^\perp$, 
so that the dimension of $(\closure \stdCone) \cap \{\stdInt_1 \}^\perp$ is at least $r-1$.

\end{proof}

\begin{proposition}\label{prop:rank2}
Let $\stdCone$ be a convex cone such that $\closure \stdCone$ is strictly convex. If $\stdCone$ is homogeneous and nonzero, then its rank is less or equal than $2$.
\end{proposition}
\begin{proof}
It is an immediate consequence of Proposition \ref{prop:h_face}. If $\stdCone$ were 
homogeneous of rank $r \geq 3$, then its closure would have a proper face of dimension at least $2$ contradicting the strict convexity.
Therefore, $r \leq 2$. 
\end{proof}

It turns out that the closure of  a homogeneous convex cone of rank $r \leq 2$ is self-dual under an appropriate inner product, which is  mentioned by Vinberg in \cite{V63}, see Section 8 of Chapter 3. 
Using that, the next result follows from the classification of symmetric cones, see Chapter 5 in \cite{FK94}. Nevertheless, we provide a direct proof that explicitly exhibits the isomorphism while avoiding the aforementioned classification result.
\begin{proposition}\label{prop:h_2}
Let $\stdCone$ be a nonzero homogeneous convex cone of rank $r \leq 2$. Then $\closure \stdCone$ is isomorphic to 
either $\Re_+, \Re^2_+$ or $\SOC{2}{n}$ for some $n$. 
\end{proposition}
\begin{proof}
Let $\alg$ be a $T$-algebra of rank $r$ such that $\stdHC = \interior \stdCone$.	
If $r = 1$, it is clear that $\closure \stdCone$ must be isomorphic to 
$\Re_+$. If $r = 2$, we consider two cases. If $\alg_{12} = \alg_{21} = \{0\}$ 
it is clear that $\closure \stdCone$ must be isomorphic to $\Re_+^2$. So now we consider 
the case where the dimension of $\alg_{12}$ is greater than zero and we identify 
$\alg_{12}$ with some $\Re^m$ with $m \geq 1$.

We first establish necessary and sufficient conditions for $a \in \alg$ to belong in 
$\stdHC$. Due to Proposition \ref{prop:diag}, we can write  
$a_{11} = \alpha \stdInt _1, a_{22} = \beta \stdInt_2$, where $\alpha = \rho _1(a_{11})$ and 
$\beta = \rho _2 (a_{22})$.

Now, $a \in \stdHC$ if and only if there is $t \in \UT_{++}$ such that $a = tt^*$. 
We 
can write $t_{11} = \gamma _1 \stdInt _1, t_{22} = \gamma _2 \stdInt_2$, where $\gamma _1 = \rho _1(t_{11})$ and 
$\gamma _2 = \rho _2 (t_{22})$. We have 
\begin{align}
tt^* & = t_{11}t_{11} + t_{12}t_{12}^* + t_{12}t_{22}  + t_{22}t_{12}^* + t_{22}t_{22} \notag \\
& = (\gamma _1 ^2 +  \rho _1(t_{12}t_{12}^*) ) \stdInt _1 + \gamma _2 t_{12}\stdInt_2  + \gamma _2  \stdInt _1 t_{12}^* + \gamma _2^2 \stdInt _2 \notag \\
& = (\gamma _1 ^2 +  \rho _1(t_{12}t_{12}^*) ) \stdInt _1 + \gamma _2 t_{12}  + \gamma _2  t_{12}^* + \gamma _2^2 \stdInt _2, \notag 
\end{align}
where the last equality follows from Axiom \ref{ax:2}. Then, comparing the expressions for $tt^*$ and $a$, we conclude that in order to have $tt^* = a$, we must have $a = a^*$ and
\begin{align*}
\gamma_2 & = \sqrt{\beta} \notag \\
t_{12} & = \frac{a_{12}}{\sqrt{\beta}} \notag \\
\gamma _1^2   & = \alpha - \frac{\rho_1(a_{12}a_{12}^*)}{\beta}. \label{eq:aux} 
\end{align*} 
Since $\gamma _1$ and $\gamma _2$ must be positive, we conclude that $a \in \stdHC$ if and only if $a = a^*$, $\alpha > 0$, $\beta > 0$ and $\alpha \beta - {\rho_1(a_{12}a_{12}^*)} > 0$. The last inequality can be expressed equivalently as 
\begin{equation}
\left(\frac{\alpha + \beta}{2} \right)^2 > \left(\frac{\alpha - \beta}{2} \right)^2 + {\rho_1(a_{12}a_{12}^*)}. \label{eq:aux4}
\end{equation}
Now, notice that we have $\rho_1(a_{12}a_{12}^*)=\inProd{a_{12}^*}{a_{12}^*} = \inProd{a_{12}}{a_{12}}$ when $a=a^*$.
Restricting $\inProd{\cdot}{\cdot}$ as an inner product on $\alg_{12}$, we can write
$\inProd{u}{v} = u^T Q^T Qv$ for each $u,v \in \alg_{12}$ with some $m \times m$ nonsingular matrix $Q$.
%Now, we observe that the function $B: \alg _{12} \times \alg _{12} \to \Re$ defined by $B(u,v) = \rho_1(uv^*)$ is a real inner product on $\alg _{12}$. The bilinearity comes from the fact that multiplication in $\alg$ is distributive and $\rho _1$ is an isomorphism. Symmetry holds because $uv^* \in \alg _{11}$, so by Proposition \ref{prop:diag}, 
%$(uv^*)^* =  uv^*$, which implies that $vu^* = uv^*$. Finally, nondegeneracy is a consequence of Axiom \ref{ax:5}. 
%Therefore, there is a nonsingular matrix $Q$ such that $u^TQ^T Qv = B(u,v) = \rho_1(uv^*)$, for all $u,v \in \alg_{12}$. 
%Substituting this in \eqref{eq:aux4},
Then, we conclude that $a \in \stdHC$ if and only if $a = a^*$, $\alpha + \beta > 0$ and 
\begin{equation}
\left(\frac{\alpha + \beta}{2} \right)^2 > \left(\frac{\alpha - \beta}{2} \right)^2 + (Qa_{12})^T(Qa_{12}),
\end{equation}
which is equivalent to the statement that $(\frac{\alpha + \beta}{2},\frac{\alpha - \beta}{2}, Qa_{12} ) \in \interior \SOC{2}{m+2}$.

Let $\mathcal{H}(\alg) = \{a \in \alg \mid a = a^*\}$ and consider the linear map $S:\mathcal{H}(\alg) \to \Re^{m+2}$ defined by 
$$S(a) := \left(\frac{\rho_1(a_{11}) + \rho_2(a_{22})}{2},\frac{\rho_1(a_{11}) - \rho_2(a_{22})}{2}, Qa_{12} \right) = \left(\frac{\alpha + \beta}{2}, \frac{\alpha - \beta}{2}, Qa_{12}\right).$$
Since $Q$ is nonsingular, $S$ is a bijection. Then, the discussion so far implies that 
$S(\stdHC) = \interior \SOC{2}{m+2}$. Therefore, $S$ is the desired isomorphism.

\end{proof}

 We thus arrive at the main result.
\begin{theorem}
Suppose that $\stdCone$ is a convex cone such that $\closure \stdCone$ is strictly convex. 
Then, it is homogeneous if and only 
if $\closure \stdCone$ is isomorphic to $\{0\}$, $\Re_+$, $\Re ^2_+$ or $\SOC{2}{n}$.	

In particular, $\SOC{p}{n}$ is not homogeneous if $p \neq 2$ and $n \geq 3$.
\end{theorem}
\begin{proof}
It is clear that $\{0\}$, $\Re_+$, $\Re ^2_+$ and $\SOC{2}{n}$ are homogeneous convex cones.
For the converse: If $\stdCone = \{0\}$ we are done. Otherwise, by Propositions \ref{prop:rank2} and \ref{prop:h_2},
$\stdCone$ must be isomorphic to $\Re_+$, $\Re ^2_+$ or $\SOC{2}{n}$. This concludes 
the first half.

Now, suppose that $\stdCone$ is homogeneous and  $\stdCone = \SOC{p}{n}$ for $p \neq 2$, $1 < p < \infty$ and $n \geq 3$.
Then $\stdCone$ is strictly convex and, therefore, must be isomorphic to one of the four cones listed above.
The only possible candidate 
is $\SOC{2}{n}$, since all the others have  dimension less or equal than $2$. 
However, the results by Gowda and Trott in \cite{GT14} imply that $\SOC{2}{n}$ and $\SOC{p}{n}$ are not isomorphic if $p \neq 2$ and $n \geq 3$, since an invariant known as  ``Lyapunov rank'' is $\frac{n^2-n+2}{2}$ for the former and $1$ for the latter. Furthermore, isomorphic cones have the same Lyapunov rank. See Section 1 and Theorem 5 in \cite{GT14} and the related paper \cite{GTa14}, for more details. This gives a contradiction, so $\stdCone$ is not homogeneous.

We remark that  Gowda and Trott already showed that $\SOC{1}{n}$ and $\SOC{\infty}{n}$ are not homogeneous for $n \geq 3$, see Theorem 7 and Section 6 of \cite{GT14}.
\end{proof}

\section*{Acknowledgements}
The authors would like to thank Prof.~Chen for kindly agreeing to send us a preliminary version of \cite{MLC17}. 
The authors would also like to thank an anonymous referee for many valuable suggestions that improved the presentation of this paper.
The second author  was partially supported by the Grant-in-Aid for Scientific Research (B) (15H02968) from Japan Society for the Promotion of Science.
\bibliographystyle{abbrvurl}
\bibliography{bib}
\end{document}